%% file: main.tex
\documentclass[11pt, reqno]{amsart}
\usepackage{amssymb}
\usepackage{amsmath}
\usepackage{mathrsfs}
\usepackage{amsfonts}
\usepackage{color}
\usepackage{amsthm}
\usepackage{graphicx}
\newtheoremstyle{remark}
  {}{}{}{}{\bfseries}{.}{.5em}{{\thmname{#1 }}{\thmnumber{#2}}{\thmnote{ (#3)}}}

\input{packages}


\input{naming}

\title[Banach-Saks property for Hölder spaces]{The weak Banach-Saks property for Hölder spaces}

\author[Górka]{Prezemysław Górka}
\author[Sanchiz]{Mauro Sanchiz$^{*}$}

\address{Faculty of Mathematics and Information Sciences\\
Warsaw University of Technology\\
Pl. Politechniki 1, 00-661 Warsaw, Poland}
\email{ przemyslaw.gorka@pw.edu.pl}
\address{Departamento de An{\'a}lisis Matem{\'a}tico y Matem\'atica Aplicada, Facultad de Matem{\'a}ticas, Universidad Complutense, 28040 Madrid, Spain}
\email{ msanchiz@ucm.es}
\thanks{$^{*}$ Supported by NAWA under the Ulam postoctoral program 2024/1/00064 and partially supported by the UNED research project: 2025/00145/001 ``Operadores, retículos y estructura de espacios de Banach''}

\begin{document}

\begin{abstract}
We investigate the weak Banach--Saks property in the setting of H\"older spaces over metric spaces. We show that, for every infinite metric space $(M,d)$ and every $\alpha \in (0,1]$, the H\"older space $C^{\alpha}(M)$ fails to have the weak Banach--Saks property.

\end{abstract}

\maketitle 

\bigskip
\section{Introduction}
A Banach space $X$ is said to have the \emph{weak Banach--Saks property} if every weakly convergent sequence $x_n \to x$ in $X$ contains a subsequence $(x_{n_k})$ whose Ces\`aro means converge in norm to $x$, that is,
\[
\Big\| \frac{1}{N} \sum_{k=1}^N x_{n_k} - x \Big\|_X \longrightarrow 0 
\quad \text{as } N \to \infty.
\]

In their seminal paper \cite{banach_sur_1930}, Banach and Saks proved that the spaces $L_p[0,1]$ (for $1 < p < \infty$) and $\ell_p$ possess the weak Banach--Saks property. They also presented an incorrect argument claiming that $L_1[0,1]$ fails to have this property. This was later corrected by Szlenk \cite{szlenk_sur_1965}, who showed that $L_1[0,1]$ is in fact weakly Banach--Saks. Furthermore, Schreier \cite{schreier_gegenbeispiel_1930} proved that $C[0,1]$ does not have the weak Banach--Saks property. Finally, Farnum \cite{farnum_banach-saks_1974} established that for a compact metric space $M$, the space $C(M)$ is weakly Banach--Saks if and only if
\[
M^{(\omega)} = \emptyset, 
\quad \text{where } 
M^{(\omega)} = \bigcap_{n=1}^{\infty} M^{(n)},
\]
and $M^{(n)}$ denotes the $n$-th derived set of $M$.

Since then, occasional works on the weak Banach--Saks property have appeared in different context on Banach spaces. More recently, quantifications of the Banach--Saks properties have been studied in \cite{bendova_quantification_2015, silber_quantification_2023}.

The main objective of this paper is to show that, for every infinite metric space $M$, the H\"older space $C^{\alpha}(M)$ does not have the weak Banach--Saks property. This result complements the existing characterizations for spaces of continuous functions and provides a natural extension of earlier work to the H\"older setting.

The paper is organized as follows. In Section~2, we recall several definitions and auxiliary results. We provide a direct proof that $\ell_\infty$ does not have the weak Banach--Saks property. As corollaries, we obtain characterizations of measure spaces $(X,\Sigma,\mu)$ and metric spaces $(M,d)$ for which $L_\infty(X,\Sigma,\mu)$ and $C_b(M)$, respectively, are weakly Banach--Saks. Section~3 contains the proof that $C^{\alpha}(M)$ has the weak Banach--Saks property if and only if the metric space $(M,d)$ is finite.

\section{Preliminaries}
Let $(M,d)$ be a metric space, and let $C_b(M)$ denote the space of all bounded continuous functions on $M$. We equip $C_b(M)$ with the supremum norm
\[
\|f\|_{C(M)} := \sup_{x \in M} |f(x)|.
\]
Fix $0 < \alpha \le 1$. A function $f \in C_b(M)$ is said to be \emph{Hölder continuous of order $\alpha$} if
\[
\rho_{\alpha, M}(f) := \sup_{\substack{x,y \in M \\ x \neq y}} 
\frac{|f(x) - f(y)|}{d(x,y)^\alpha} < \infty.
\]
We define
\[
C^\alpha(M) := \{ f \in C_b(M) : \rho_{\alpha, M}(f) < \infty \},
\]
and equip it with the norm
\[
\|f\|_{C^\alpha (M)} := \|f\|_{C(M)} + \rho_{\alpha, M}(f).
\]
The spaces $C_b(M)$ and $C^\alpha(M)$ are Banach spaces. 

By the definition of the completion of a metric space, we obtain the following.
\begin{prop} \label{iso}
    Let $(M,d)$ be a metric space, and let $(\widehat{M},\widehat{d})$ be its completion. Fix $\alpha \in (0,1]$. Then the following $C^\alpha(\widehat{M}) \cong C^\alpha(M)$ isometric isomorphism holds.
\end{prop}
\begin{proof}
Let $(\widehat{M},\widehat{d})$ denote the completion of $(M,d)$, and let 
$i: M \to \widehat{M}$ be an isometric embedding with dense image. 
Define the operator
\[
I: C^{\alpha}(\widehat{M}) \to C^{\alpha}(M)
\]
by
\[
I(f)(x) = f(i(x)), \qquad \text{for all } f \in C^{\alpha}(\widehat{M}),\ x \in M.
\]

It is immediate that $I$ is linear and injective. Moreover, for every 
$f \in C^{\alpha}(\widehat{M})$,
\[
\|I(f)\|_{C^{\alpha}(M)} \le \|f\|_{C^{\alpha}(\widehat{M})}.
\]

We now prove that $I$ is surjective. Let $g \in C^{\alpha}(M)$ be given and define 
$\widetilde{f}: i(M) \to \mathbb{R}$ by
\[
\widetilde{f}(\hat{x}) = g(i^{-1}(\hat{x})).
\]
For any $\hat{x}, \hat{y} \in i(M)$ we obtain
\[
|\widetilde{f}(\hat{x}) - \widetilde{f}(\hat{y})|
= |g(i^{-1}(\hat{x})) - g(i^{-1}(\hat{y}))|
\le \rho_{\alpha,M}(g)\,d^{\alpha}(i^{-1}(\hat{x}),i^{-1}(\hat{y}))
=\rho_{\alpha,M}(g)\,\widehat{d} ^{\alpha}(\hat{x},\hat{y}),
\]
and
\[
|\widetilde{f}(\hat{x})| \le \|g\|_{C(M)}.
\]
Thus, $\widetilde{f}$ is Hölder continuous on the dense subset 
$i(M) \subset \widehat{M}$ with the same Hölder seminorm and supremum bound as $g$.
Since $i(M)$ is dense in $\widehat{M}$, the function $\widetilde{f}$ admits 
a unique continuous extension $f: \widehat{M} \to \mathbb{R}$. 
Furthermore, this extension satisfies
\[
\rho_{\alpha,\widehat{M}}(f) \le \rho_{\alpha,M}(g),
\qquad
\|f\|_{C(\widehat{M})} \le \|g\|_{C(M)}.
\]
Hence $f \in C^{\alpha}(\widehat{M})$ and, by construction, $I(f)=g$. 
Consequently,
\[
\|f\|_{C^{\alpha}(\widehat{M})}
\le
\|I(f)\|_{C^{\alpha}(M)}.
\]
This completes the proof.
\end{proof}

Let \(X\) and \(Y\) be Banach spaces with norms \(\|\cdot\|_X\) and \(\|\cdot\|_Y\), respectively. 
We say that \(Y\) embeds into \(X\), denoted \(Y \hookrightarrow X\), if there exists a linear map 
\(T: Y \to X\) and a constant \(C > 0\) such that
\[
C^{-1} \|y\|_Y \leq \|Ty\|_X \leq C \|y\|_Y \quad \text{for all } y \in Y.
\]

\begin{prop}\label{propemb}
Let \(X\) and \(Y\) be Banach spaces. 
If \(X\) is weakly Banach-Saks and \(Y \hookrightarrow X\), 
then \(Y\) is also weakly Banach-Saks.
\end{prop}
\begin{proof}
Let \(T: Y \to X\) be a linear operator. Assume that there exists a constant \(C>0\) such that
\begin{eqnarray}\label{zbi}
C^{-1}\|y\|_{Y} \le \|Ty\|_{X} \le C\|y\|_{Y}
\quad \text{for all } y \in Y.
\end{eqnarray}
Let \((y_n)\subset Y\) be a sequence converging weakly to \(y \in Y\). By linearity and continuity of \(T\), it follows that 
\(T(y_n) \rightarrow T(y)\) weakly in \(X\).
Since \(X\) has the weak Banach--Saks property, there exists a subsequence 
\((y_{n_k})\) such that
\[
\frac{1}{N}\sum_{k=1}^{N} T(y_{n_k}) - T(y)
\longrightarrow 0
\quad \text{in } X.
\]
Using the first inequality in \eqref{zbi}, we deduce that
\[
\frac{1}{N}\sum_{k=1}^{N} y_{n_k} - y
\longrightarrow 0
\quad \text{in } Y.
\]
This completes the proof.

\end{proof}
\begin{thm}\label{zan}
    $\ell_\infty$ is not weakly Banach-Saks.
\end{thm}
\begin{proof}
Although the theorem can be deduced using advanced results, we present a more direct approach for clarity. 
Indeed, since \(C[0,1]\) is separable, Theorem 2.5.7 in \cite{bookA} provides an embedding $
C[0,1] \hookrightarrow \ell_\infty,$
and, by Schreier's Theorem \cite{schreier_gegenbeispiel_1930}, the space \(C[0,1]\) is not weakly Banach-Saks. 
Consequently, Proposition~\ref{propemb} implies that \(\ell_\infty\) is not weakly Banach-Saks. 
While this argument is valid, it relies on relatively heavy machinery; 
to provide a self-contained and more transparent exposition, we shall give the following direct proof following the ideas of Schreier and Farnum.

Let $\mathcal{A} \subset 2^{\mathbb{N}}\setminus \{\emptyset\}$ be the family of maximal Schreier sets \cite{BEANLAND_GOROVOY_HODOR_HOMZA_2024}, defined by
\[
    \mathcal{A}=\{A \subset \mathbb{N}: \vert A\vert=\min(A) \}\setminus\{\emptyset\}.
\] Clearly, $\mathcal{A}$ is a numerable set. Take $T:\mathbb{N}\rightarrow\mathcal{A}$ an enumeration of $\mathcal{A}$, and define the sequence $(u_k)$ in $\ell_\infty$ as follows
$$
u_k(i):=\begin{cases}
    1 \text{ if } k\in T(i)\\
    0 \text{ otherwise.}
\end{cases}
$$
\begin{lemma} \label{jeden}
    The sequence $(u_k)$ is weakly converging in $\ell_\infty$ to $0$.
\end{lemma}
\begin{proof}
We shall need the following criterion for weak convergence in $\ell_{\infty}$.
\begin{prop}[\cite{toland_dual_2020}, Corollary 8.11, rewritten]
    Let $(u_k)$ be a sequence in $\ell_{\infty}$ with $u_k(i) \rightarrow 0$ as $k \rightarrow \infty$ for each $i \in \mathbb{N}$.
    If for every $\alpha>0$ and for all strictly increasing natural sequences $(k_j)$, $(i_n)$ and $(J_n)$, there exists $n\in\mathbb{N}$ and $j\in\{1,2,...,J_n\}$ such that $\vert u_{k_j}(i_n)\vert\leq \alpha$, then $u_k \rightarrow 0$ weakly in $\ell_{\infty}$.
\end{prop}
    To apply the above proposition we fix $\alpha>0$, and strictly increasing sequences of natural numbers $(k_j)$, $(i_n)$ and $(J_n)$. Take $n>k_1$. If $k_1\not\in T(i_n)$, then $u_{k_1}(i_n)=0<\alpha$ and the proof ends, so suppose that $k_1\in T(i_n)$. This forces $\vert T(i_n)\vert=\min(T(i_n))\leq k_1$. Now, since $k_1<n\leq J_n$ and $k_1<k_2<...<k_{J_n}$, there must exist at least some $j\in\{1,...,J_n\}$ such that $k_j\not\in T(i_n)$.
    Hence, $\vert u_{k_j}(i_n)\vert=0<\alpha$. This concludes that $u_k\rightarrow0$ weakly in $\ell_{\infty}$.
\end{proof}

\begin{lemma}\label{dwa}
    The sequence $(u_k)$ has no subsequences whose Cèsaro means converge in $\ell_\infty$ to $0$. 
\end{lemma}

\begin{proof}
        Let us fix any strictly increasing sequence $(k_j)$ of natural numbers,
    observe that $k_{N+1}>N$. Let 
    \[
    A_N:=\{k_{N+1},...,k_{N+N},k_{N+N+1},...,k_{{N+k_{N+1}}}\}.
    \]
    Then the set $\{k_{N+1}, k_{N+2},..., k_{N+N}\}$ is contained in $A_N$. Since $A_N \in \mathcal{A}$, there exists $i_0$ with $T(i_0)=A$, and in consequence $u_{k_j}(i_0)=1$ for every $j\in\{N+1,...,N+N\}$. This leads to
    $$
    \left\| \frac{1}{2 N} \sum_{j=1}^{2N} u_{k_j}\right\|_{\ell_\infty}\geq\frac{1}{2 N} \sum_{j=1}^{2N} u_{k_j}(i_0)\geq \frac{1}{2 N} \sum_{j=N+1}^{2N} u_{k_j}(i_0)=\frac{1}{2}.
    $$
    In conclusion, since for any $N$ we have $\lVert \frac{1}{2 N} \sum_{j=1}^{2N} u_{k_j}\rVert_{\ell_\infty}\geq \frac{1}{2}$, the subsequence $(u_{k_j})$ is not Cèsaro null.
\end{proof}
Now, by Lemma \ref{jeden} and Lemma \ref{dwa} the proof of the theorem follows.
\end{proof}
As corollaries, we get the following two results.
\begin{thm} Let $(X,\Sigma,\mu)$ be measure space, then $L_\infty(X,\Sigma,\mu)$ is not weakly Banach-Saks if and only if there exists a sequence $(A_n)$ of disjoint measurable sets of positive measure.
\end{thm}

\begin{proof}

    $(\Longrightarrow)$ By Proposition~\ref{propemb} and Theorem~\ref{zan}, it is sufficient to prove that $\ell_\infty \hookrightarrow L_\infty(X,\Sigma,\mu).$

Let \((A_n)\) be a sequence of pairwise disjoint measurable sets of positive measure. 
Define an operator \( T : \ell_\infty \to L_\infty(X,\Sigma,\mu) \) by
\[
T\big(a\big) = \sum_{n=1}^{\infty} a(n) \chi_{A_n}, 
\qquad a \in \ell_\infty.
\]
Then \(T\) is a linear isometry. Consequently, \(\ell_\infty\) embeds isometrically 
into \(L_\infty(X,\Sigma,\mu)\), as claimed.

    $(\Longleftarrow)$ Assume, for the sake of contradiction, that the statement is false. 
Then there exists a natural number \( N \in \mathbb{N} \) and measurable, 
pairwise disjoint sets \( A_1, \dots, A_N \) of positive measure such that $
\mu\left( X \setminus \bigcup_{i=1}^{N} A_i \right) = 0,$
and, moreover, for every \( i \in \{1, \dots, N\} \) and every measurable subset 
\( A \subset A_i \), we have $\mu(A)\,\mu(A_i \setminus A) = 0.$
Then any $f \in L_\infty(X, \Sigma, \mu)$ must be constant almost everywhere on each $A_i$ for $i \in \{1, \dots, N\}$. 
Indeed, if $f \in L_\infty(X,\Sigma,\mu)$ were not constant on $A_i$ for some $i$, then $\operatorname{ess\,sup}_{A_i} f > \operatorname{ess\,inf}_{A_i} f,$
which yields two disjoint measurable subsets of $A_i$ of positive measure — a contradiction. Hence every $f$ is constant on each $A_i$. Thus $L_\infty(X,\Sigma,\mu)$ is finite-dimensional, and therefore it has the weak Banach--Saks property.
\end{proof}
\begin{thm}Let $(M,d)$ be a metric space. Then the following statements hold:
\begin{enumerate}
    \item If $(M,d)$ is compact, the space $C_b(M)$ is weakly Banach--Saks if and only if $M^{(\omega)} = \emptyset$.
    \item If $(M,d)$ is not compact, the space $C_b(M)$ is not weakly Banach--Saks.
\end{enumerate}
\end{thm}
\begin{proof}
Since $(1)$ is just the main result of Farnum \cite{farnum_banach-saks_1974} we shall prove $(2)$. Since $(M,d)$ is not compact, there exists a sequence $(x_n,\varepsilon_n) \subset X \times (0,\infty)$ such that the sequence $(x_n)$ has no convergent subsequence, $\varepsilon_n \to 0$, and
\[
B(x_i,\varepsilon_i) \cap B(x_j,\varepsilon_j) = \emptyset 
\quad \text{for all } i \neq j.
\]
For each $n \in \mathbb{N}$, let $\varphi_n : M \to [0,1]$ be a continuous function satisfying the following conditions\footnote{For example, one may define $\varphi_n(x)=\max\left\{1-\frac{d(x,x_n)}{\varepsilon_n},\,0\right\}$.}
\[
\varphi_n(x_n) = 1 
\quad \text{and} \quad 
\varphi_n \equiv 0 \text{ on } M \setminus B(x_n,\varepsilon_n).
\]
Define the linear operator $T \colon \ell_\infty \to C_b(M)$ by
\[
T\big(a\big) = \sum_{n=1}^{\infty} a(n) \varphi_n, 
\qquad a \in \ell_\infty.
\]
It follows from the properties of the sequence $(x_n,\varepsilon_n)$ that, for every $a \in \ell_\infty$, the map $x \mapsto T(a)(x)$ is continuous. Since $T$ is a linear isometry, it follows from Theorem~\ref{zan} and Proposition~\ref{propemb} that $C_b(M)$ fails to be weakly Banach--Saks.

\end{proof}

\section{Main result}

\begin{thm} Let $(M,d)$ be a metric space and $\alpha \in (0,1]$, then $C^\alpha(M)$ is weakly Banach-Saks if and only if $M$ is finite.
\end{thm}

\begin{proof}
First, we recall the following lemma from \cite{cuth_structure_2016}\footnote{It is not explicitly stated in that article as we write it, but it is proven in the proof of its Theorem 3.2, which mainly shows the existence of points a sequence of points $(x_n,y_n)$ satisfying the assumptions of its Lemma 3.1, the properties we are taking.}.

\begin{lemma}\label{ciagi}
Let $(M,d)$ be an infinite complete metric space. Then there exist a constant $K>0$ and a sequence of pairs of points $(x_n,y_n) \subset M$ such that:
\begin{enumerate}[label=(\roman*)]
    \item $x_n \neq y_n$ for all $n \in \mathbb{N}$;
    \item $x_m \notin B(y_n, K \, d(x_n,y_n))$ for all $n,m \in \mathbb{N}$ (in particular, this implies $K \leq 1$);
    \item $B(y_n, K \, d(x_n,y_n)) \cap B(y_m, K \, d(x_m,y_m)) = \emptyset$ for all $n \neq m$.
\end{enumerate}
\end{lemma}
Clearly, if $M$ is finite, then $C^{\alpha}(M)$ is finite-dimensional, and thus it is weakly Banach-Saks.
Let $(M, d)$ be an infinite metric space. By Proposition~\ref{iso}, the space $C^\alpha(M)$ is isometrically isomorphic to $C^\alpha(\widehat{M})$, where $(\widehat{M},\widehat{d})$ denotes the completion of the metric space $(M,d)$. Therefore, in view of Proposition~\ref{propemb}, we may assume without loss of generality that $(M,d)$ is complete.

By Proposition \ref{propemb} and Theorem \ref{zan} it is enough to show that $\ell_{\infty} \hookrightarrow C^\alpha(M)$.
Let $K>0$ and let $(x_n,y_n)$ be a sequence of pairs of points in $M$ from Lemma \ref{ciagi}.
   Let us define the sequence of functions $f_n :M \rightarrow \mathbb{R}$ as follows 
    $$
    f_n(x):=\max\left\{ \min\left\{1,d^{\alpha}(x_n,y_n)-\frac{d^\alpha(x,y_n)}{K^\alpha}\right\}, 0 \right\},  \,\, x \in M
    $$
     It is trivial that 
     \[
     \lVert f_n\rVert_{C(M)}\leq 1.
     \]
     Moreover, we have
     \begin{eqnarray} \label{oszacowania}
       \rho_{\alpha, M}(f_n)\leq \frac{1}{K^\alpha}.
     \end{eqnarray}
     Indeed, since the maps $\xi \mapsto \max\{\xi, 0\}$ and $\xi \mapsto \min\{1, \xi\}$ are $1$-Lipschitz, we have 
     
    \begin{align*}
        \rho_{\alpha, M}(f_n)={}    &   \sup_{x\neq y} \frac{\vert f_n(x)-f_n(y)\vert}{d^\alpha(x,y)}\\
        \leq &  \sup_{x\neq y}\frac{\left|d^{\alpha}(x_n,y_n)-\frac{d^\alpha(x,y_n)}{K^\alpha} -\left(d^{\alpha}(x_n,y_n)-\frac{d^\alpha(y,y_n)}{K^\alpha}\right)  \right|}{d^\alpha(x,y)}\\
        = & \sup_{x\neq y} \frac{\left|\frac{d^\alpha(x,y_n)}{K^\alpha} -\frac{d^\alpha(y,y_n)}{K^\alpha}  \right|}{d^\alpha(x,y)}.
    \end{align*}
    Next, we use  the inequality $\vert a^\alpha-b^\alpha\vert\leq \vert a-b\vert^\alpha$ for every $0<\alpha\leq 1$ and $a,b>0$ to compute
    \begin{align*}
        \rho_{\alpha, M}(f_n)\leq{}    &   \sup_{x\neq y}\frac{\vert d(x,y_n)-d(y,y_n)\vert^\alpha}{K^\alpha \cdot d^\alpha(x,y)}\\
    \leq{}  &   \sup_{x\neq y} \frac{\vert d(x,y)\vert^\alpha}{K^\alpha \cdot d^\alpha(x,y)}=\frac{1}{K^\alpha},
    \end{align*}
    and this finishes the proof of (\ref{oszacowania}).

 For each $n \in \mathbb{N}$, let us denote $B_n:=B(y_n,K\cdot d(x_n,y_n))$. From the very definition of $f_n$ we have $supp (f_n):=\{x \in M: f_n (x) \neq 0\} = B_n$.
 Observe that, thanks to property $(iii)$, we have $supp(f_n)\cap supp(f_m)=\emptyset$ for $n\neq m$. Hence, for every $x\in M$, there exists at most one $n\in\mathbb{N}$ such that $f_n(x)\neq 0$. Denote $n(x)$ the unique $n$ with $f_n(x)\neq 0$ (or $n(x)=1$ if there is no such $n$). We are ready now to show that $\ell_{\infty}  \hookrightarrow C^{\alpha}(M)$. For this purpose we define the operator $   
    T \colon \ell_{\infty}  \rightarrow C^{\alpha}(M) $ in the following manner
     \begin{align*}
     T(a)(x):=a(n(x))\cdot f_{n(x)}(x), \,\, a\in \ell_{\infty}, x \in M.
    \end{align*}
     It is clearly lineal, and we will show that 
    \begin{eqnarray}\label{trzy}
    \lVert a\rVert_{\ell_\infty}\leq \lVert T(a)\rVert_{C^\alpha(M)}\leq \left(\frac{2}{K^\alpha}+1\right)\lVert a\rVert_{\ell_\infty}, \,\, \text{for}\,\, a\in \ell_\infty.
    \end{eqnarray}
    Let us fix $a\in\ell_\infty$. 
    On the one hand, fix $x,y\in M$. If $n(x)=n(y)$, then by (\ref{oszacowania}) we have
    \begin{align*}
    \vert T(a)(x)-T(a)(y)\vert =& \vert a(n(x)) (f_{n(x)}(x)-f_{n(x)}(y))\vert \\ \leq & \lVert a\rVert_{\ell_\infty} \cdot \rho_{\alpha, M} (f_{n(x)})d^\alpha(x,y)\leq \lVert a\rVert_{\ell_\infty} \cdot \frac{d^\alpha(x,y)}{K^\alpha}.
    \end{align*}
    Hence, assume $n(x)\neq n(y)$. Thus we have $f_{n(x)}(y)=f_{n(y)}(x)=0$, and consequently, by (\ref{oszacowania}), we obtain 
    \begin{align*}
    \vert T(a)(x)-T(a)(y)\vert 
    ={} &    \vert a(n(x))\cdot f_{n(x)}(x) - a(n(y))\cdot f_{n(y)}(y)\vert \\
    = &\vert a(n(x))\cdot (f_{n(x)}(x)-f_{n(x)}(y)) - a(n(y))\cdot (f_{n(y)}(y)-f_{n(y)}(x))\vert\\
    \leq{}  &   |a(n(x))| \vert f_{n(x)}(x)-f_{n(x)}(y) \vert + |a(n(y))| \vert f_{n(y)}(y)-f_{n(y)}(x)\vert\\
        \leq{}  &   |a(n(x))| \rho_{\alpha, M} (f_{n(x)})d^\alpha(x,y) + |a(n(y))|\rho_{\alpha, M} (f_{n(x)})d^\alpha(x,y) \\
    \leq & \lVert a\rVert_{\ell_\infty} \cdot\frac{2 d^\alpha(x,y)}{K^\alpha}.
    \end{align*}
    Therefore,
    \[
     \rho_{\alpha, M}(T(a)) \leq \frac{2}{K^\alpha} \|a\|_{\ell_\infty},
    \]
    and it is clear that 
    \[
    \lVert T(a)\rVert_{C(M)}\leq\lVert a\rVert_{\ell_\infty} \cdot \sup_n\lVert f_n\rVert_{C(M)}\leq\lVert a\rVert_{\ell_\infty}.
    \]

    On the other hand, for every $k\in\mathbb{N}$, $n(y_k)=k$ and $x_k\not\in \bigcup_{l=1}^\infty B_l$. Thus, for every $l \in \mathbb{N}$ we have  $f_l(x_k)=0$ and
    \begin{eqnarray*}
    \vert T(a)(x_k)- T(a)(y_k) \vert&=&|T(a)(y_k)|=|a(n(y_k))\cdot f_{n(y_k)}(y_k)|\\
    &= &|a(k) \cdot f_k(y_k)|=|a(k)| \cdot \min\{1,d^{\alpha}(x_k,y_k)\}.
    \end{eqnarray*}
    This shows that either $|a(k)|\leq\lVert T(a)\rVert_{C(M)}$ or $|a(k)|\leq \rho_{\alpha, M}(T(a))$. Since $k$ is arbitrary, we get
    \[
    \lVert a\rVert_{\ell_\infty}\leq \rho_{\alpha, M}(T(a)) + \lVert T(a)\rVert_{C(M)} = \lVert T(a)\rVert_{C^\alpha(M)},
    \]
    this finishes the proof of (\ref{trzy}) and the whole proof follows.
\end{proof}

\printbibliography

\end{document}

%% file: packages.tex
\usepackage[utf8]{inputenc} 

\usepackage[width=16cm,height=24cm]{geometry} 

\usepackage[backend=biber,
style=numeric, url=false,doi=false, isbn=false 
, maxnames=10 
, sorting=nyt]{biblatex}
\usepackage[autostyle=true]{csquotes} 

\usepackage{amsfonts,amsmath,amssymb,amsthm}
\usepackage{graphics,graphicx,epsfig}

\usepackage{comment}
\usepackage{color}
\usepackage{bbm}
\usepackage{bbold}
\usepackage{enumitem}
\usepackage{appendix}
\usepackage{tikz, tikz-cd}

%% file: naming.tex
\newtheorem{thm}{Theorem}[section]

\newtheorem{lemma}[thm]{Lemma}
\newtheorem{prop}[thm]{Proposition}

\newtheorem*{conjecture*}{Conjecture}

